\documentclass[12pt,a4paper,english]{smfart}
\usepackage[english]{babel}
\usepackage{smfthm,mathabx}
\usepackage{amssymb}
\usepackage{amsmath}
\usepackage{amsthm}
\usepackage{amsfonts}
\usepackage{graphicx}

\usepackage{appendix}

\usepackage{euscript,mathrsfs}
\usepackage[utf8]{inputenc} 
\usepackage{longtable}
\usepackage{dsfont}

\usepackage[all,cmtip]{xy}
\usepackage{alltt}
\usepackage{footmisc}

\usepackage{calrsfs}

\xyoption{all}

\usepackage{float}
\usepackage{fullpage}

\usepackage{color}
\definecolor{gray}{gray}{0.85}
\usepackage{multirow}

\newtheorem{Theorem}{Theorem}

\newtheorem*{conjecture}{Conjecture}
\newtheorem*{Corollary}{Corollary}
\newtheorem{Remark}{Remark}

\author{Christian Maire}
\address{FEMTO-ST Institute \\ Universit\'e Bourgogne Franche-Comt\'e, 15B Avenue des Montboucons  \\25030 Besan\c con Cedex\\ France }
\email{christian.maire@univ-fcomte.fr}

\author[Marine Rougnant]{Marine Rougnant}
\address{ Laboratoire de Math\'ematiques de Besançon\\ Universit\'e Bourgogne Franche-Comt\'e \\ UFR Sciences et Techniques \\ 16 route de Gray \\25030 Besan\c con Cedex\\ France }
\email{marine.rougnant@univ-fcomte.fr }

\title{A note on  $p$-rational fields and the abc-conjecture}

\thanks{The authors thank Bruno Angl\`es for pointing them the work of Ichimura. They also thank Georges Gras for constructive observations, Jean-Fran\c cois Jaulent for encouragements, useful comments and the extremely attentive reading, Gebhard B\"ockle for the exchanges concerning \cite{Boeckle-andall} and Zakariae Bouazaoui for his interest in this work. They also want to thank  the  anonymous referees for their careful works and helpful remarks. 
The authors were partially supported by the ANR project FLAIR (ANR-17-CE40-0012). CM was also supported by the EIPHI Graduate School (ANR-17-EURE-0002)}

\newcommand{\Q}{\mathbb{Q}}
\newcommand{\F}{\mathbb{F}}
\newcommand{\Z}{\mathbb{Z}}

\def\Ss{\mathbb{S}}

\def\N{{\rm N}}

\newcommand{\C}{\mathbb{C}}

\def\G{{\rm G}}

\def\deg{{\rm deg}}
\def\1{\mathds{1}}

\def\p{{\mathfrak p}}

\def\Ind{{\rm Ind}}

\def\log{{\rm log}}

\def\mod{{\rm mod}}
\def\Rad{{\rm Rad}}

\def\O{{\mathcal O}}

\def\E{{\mathcal E}}
\def\U{{\mathcal U}}

\def\T{{\mathcal T}}

\def\fq{{\mathbb F}}
\def\L{{\rm L}}
\def\K{{\rm K}}

\def\Gal{{\rm Gal}}
\def\Ind{{\rm Ind}}

\def\p{{\mathfrak p}}

\date\today

\begin{document}

\maketitle

\begin{abstract} In this short note we confirm the relation between the generalized $abc$-conjecture and the $p$-rationality of  number fields. Namely, we prove  that given  $\K/\Q$  a real quadratic extension or an imaginary $S_3$-extension, if the  generalized $abc$-conjecture holds in~$\K$, then there exist at least $c\  \log X$ prime numbers $p \leq X$ for which~$\K$ is $p$-rational, here $c$ is some nonzero constant depending on~$\K$. The real quadratic  case was recently suggested by  B\"ockle-Guiraud-Kalyanswamy-Khare.
\end{abstract}


\section*{Introduction}

Let $\K$ be a number field and let $p$ be a prime number. To simplify, we assume $p$ odd.
Denote by $\K_p$ the maximal pro-$p$-extension of $\K$ unramified outside $p$; put $G_p:=\Gal(\K_p/\K)$.

By class field theory, the pro-$p$ group $\G_p$ is finitely generated and one knows, since Shafarevich and Koch, that moreover $\G_p$ is finitely presented (meaning that $H^2(\G_p,\F_p)$ is finite).
In fact, $\G_p$ may be pro-$p$ free, for example when $\K=\Q$, or when $\K$ is an imaginary quadratic field (when $p>3$) and $p$ doesn't divide the class number of $\K$, or when $\K=\Q(\zeta_p)$ for $p$ regular primes, etc.

A number field $\K$ for which $\G_{p}$ is pro-$p$  free is called {\it $p$-rational} (\cite{Movahhedi-phd}).
Observe that~$\K$ is $p$-rational if and only if the Leopoldt conjecture holds for $\K$ at $p$ and the torsion $\T_p$ of the abelianization $\G_p^{ab}$ of $\G_p$ is trivial (see \cite{Nguyen2}, or \cite[Chapter X, \S 3]{NSW}).

\medskip

The  study of  $\T_p$ and  of the $p$-rationality started in the beginning of the 80's with Gras, Nguyen Quang Do, Movahhedi, Jaulent, and their students. Since the literature is rich: see for example \cite{Mova-Nguyen}, \cite{Movahhedi},  \cite{Gras-Jaulent},  \cite{Jaulent-Nguyen}, \cite{Movahhedi-phd}, \cite{Jaulent-Sauzet}, \cite{Sauzet}, \cite{Gras-Sram} etc. See also  \cite[Chapitre IV, \S 3 and \S 4]{Gras-livre} for a well-detailed  presentation of  $\T_p$, of the  Leopoldt conjecture and of 
 $p$-rational fields. In the spirit of our paper, let us mention here the works of
Byeon \cite {Byeon} and Assim-Bouazzaoui~\cite{Assim-Bouazzaoui} where they showed the infiniteness of $3$ and $5$-rational real quadratic fields.

Let us also precise at this level that a recent series of papers in different topics in number theory showed the interest of $p$-rational fields:  Goren \cite{Goren}, Greenberg \cite{greenberg},  B\"ockle-Guiraud-Kalyanswamy-Khare \cite{Boeckle-andall}, David-Pries \cite{David-Pries},
 Hajir-Maire \cite{Hajir-Maire}, Hajir-Maire-Ramakrishna~\cite{Hajir-Maire-Ramakrishna},  etc.
 
\medskip
 
 Assuming Leopoldt conjecture (for $\K$ at $p$), the $p$-rationality of $\K$ is therefore  equivalent to the nullity of  $\T_p$. Observe that $\T_p \simeq H^2(\G_p,\Z_p)^*$ for a cohomological point of view (see \cite{Nguyen}).
When the  $p$-Sylow of the class group of $\K$ is trivial,   the quantity $\T_p$ is isomorphic  to the torsion of the quotient of the units of the $p$-adic completions $\K_v$ of $\K$ by the closure of the global units. Moreover, if we assume that no $\K_v$ contains the $p$-roots of the unity (which is always the case when  $p>[\K:\Q]+1$),  then the triviality of $\T_p$ is equivalent to  the triviality of the {\it normalized $p$-adic  regulator} defined by Gras  \cite[Definition 5.1]{Gras-IJNT}. 
Recently,    Gras \cite{Gras-BS}, \cite{Gras-algo}, Pitoun-Varescon \cite{Pitoun-Varescon},  Barbulescu-Ray  \cite{Barbulescu-Ray} published a series of papers more concentrated on the computations  of $\T_p$, and on some heuristics.
In \cite[Conjecture 8.11]{Gras-CJM},  Gras proposed the following conjecture:

\begin{conjecture}[Gras]\label{conjecture} Let  $\K$ be a number field. Then for large $p$,  $\K$ is $p$-rational.  
\end{conjecture}

This conjecture is in the same spirit of the Wieferich prime numbers problem. Indeed, given an odd prime number $p$, to compute the $p$-valuation of  $2^{p-1}-1$ is equivalent to compute the normalized $p$-adic  regulator of the   $2$-units of $\Q$. In particular, in this case the nontriviality of the normalized $p$-adic  regulator  is equivalent for $p$ to verify  the congruence $2^{p-1} \equiv 1 (\mod \ p^2)$. 

\medskip

In \cite{Silverman} Silverman showed how the  Wieferich prime numbers are related to the $abc$-conjecture. Let us be more precise.
Given an integer $\alpha \in \Q^\times\backslash \{\pm 1\}$,   Silverman proved  that if the $abc$-conjecture holds then as $X \rightarrow \infty$
$$\#\{{\rm prime \ number} \ p, \ p\leq X, \	\alpha^{p-1} \nequiv 1 (\mod \ p^2)\} \geq c\  \log X,$$
where $c>0$ is some absolute constant. See also \cite{Graves-Murty}, and \cite{SS} for a generalization of Wieferich primes in number fields.
 
\medskip

Observe now that the generalized $abc$-conjecture has already been used in  the context of Iwasawa theory. Indeed in \cite{Ichimura} Ichimura gave a relationship between the Greenberg conjecture and the $abc$-conjecture.  A consequence of his work is that, for example,  for any real quadratic  field $\K$ if the generalized $abc$-conjecture holds in $\K$, then  the set of primes $p$ for which $\K$ is $p$-rational, is infinite. See also  \cite{Boeckle-andall}. 

\medskip

 The goal of our work is  to precise the quantity of such primes $p$, greatly inspired by the computations of Silverman.
 
 \
 
Our main result involves the isotypic subspaces $\T_p^\chi$ of $\T_p$. 
Let us observe here that
the authors studied previously in \cite{Maire-Rougnant} such cutting and the arithmetic consequences of the nullity of some $\T_p^\chi$.

\medskip

 Let $\K/\Q$ be a Galois extension of Galois group $\G$. Let us fix an odd prime number $p \nmid \# \G$. For an irreducible $\Q_p$-character $\psi$ of $\G$, let $r_\psi(E_\K)$ be the $\psi$-rank of $\Q_p \otimes E_\K$, where   $E_\K$ denotes the units of the ring of integers $\O_\K$ of $\K$.
 Let us also  cut $\T_p$ by its isotypic subspaces $\T_p^\psi$, and 
 denote  by $r_\psi(\T_p)$ the $\psi$-rank of $\T_p$.
 Observe that, assuming Leopoldt conjecture,   the number field $\K$ is $p$-rational if and only if $r_\psi(\T_p)=0$ for all irreducible $\Q_p$-characters~$\psi$. Moreover we will see that for $p\gg 0$, $r_\psi(\T_p) \leq r_\psi(E_\K)$ for all $\psi$.
 
 We will then focus on some special units $u$ of $E_\K$: we denote by $\Ss$ the set of algebraic integers $u \in \overline{\Q}$ having no conjugate on the unit circle.
 
 Here we prove:
 
\medskip
 
\begin{Theorem} \label{theo_main}
Let $\K/\Q$ be a Galois extension of Galois group $\G$ and let   $\chi$ be an irreducible $\Q$-character of $\G$ such that  the $\chi$-component of  $\Q\otimes E_\K$ contains some unit $u \in \Ss$.  If the generalized $abc$-conjecture holds for $\K$, then  as $X\rightarrow \infty$ $$\#\{ {\rm prime \ number} \ p\leq X, \ r_\psi(\T_p)< r_\psi(E_\K)  \ {\rm for \ some \ } {\rm irred. \ } \Q_p {\rm -char. \ }  \psi|\chi\} \geq  c \ \log X,$$ for some constant $c>0$ depending on $\K$.
\end{Theorem}

\medskip (Of course, in Theorem \ref{theo_main} one considers only prime numbers $p \nmid \#\G$.) 
As consequence we obtain the following result (the real quadratic case was suggested in \cite{Boeckle-andall}):

\medskip

\begin{Corollary} \label{theo_real}
Let $\K/\Q$ be a real quadratic field or an imaginary $S_3$-extension. If the generalized $abc$-conjecture holds for $\K$,  then as $X \rightarrow \infty$
$$\#\{ {\rm prime \ number} \ p \leq X, \ \K \ {\rm is } \ p{\rm -rational}\} \geq  c \ \log X,$$ for some constant $c>0$ depending on $\K$.
\end{Corollary}

\begin{Remark}
 It is well known that Leopoldt conjecture holds in the situations of Corollary, but we don't assume Leopoldt conjecture in Theorem \ref{theo_main}.
\end{Remark}

Let us add one additionnal remark about the units in  $\Ss$.

\begin{Remark} \label{rema_Pisot} 
 The  following observations will be useful for us: 
\begin{itemize}
\item[$-$] an unit $u \neq \pm 1$ for which all the conjugates are  real is in $\Ss$;
\item[$-$]  every cubic field contains some unit $u \in \Ss$;
\item[$-$] Pisot numbers are in $\Ss$.
\end{itemize}
See also \cite{Bertin} on the abundance of Pisot units.
\end{Remark}

Our work contains two  sections. In the first one, we introduce the objects we need. In the second section, we give the proofs of our results.


\section{The objects} \label{section1}
We start with a Galois extension $\K/\Q$ of degree $m$ and Galois group $\G$. We denote by~$\N$ the norm in $\K/\Q$.

Let $\O_\K$ be the ring of integers of $\K$, $E_\K$ be the units of $\O_\K$, and $\mu_\K$ be the group of the roots of the unity of $\K$.

 \medskip 
 
 Let $p$ be an {\it odd prime} number. In all that will follow, we suppose    that:
\begin{itemize} 
\item[$(i)$] $p\nmid \#\G$,
\item[$(ii)$] $p$ is unramified in $\K/\Q$,
\item[$(iii)$] $p$ does not divide the class number $h_\K$ of $\K$.
\end{itemize}

One excludes this way only a {\it finite set} of prime numbers $p$. In particular, there exists an explicit prime number $p_0$ such that  every  $p>p_0$ satisfies $(i), (ii)$ and $(iii)$.

\subsection{$p$-rational fields and isotypic components}

\subsubsection{}
Let $S_p$ be the set of places of $\K$ above $p$. For $v \in S_p$, denote by $\K_v$ the completion of $\K$ at $v$, by $\O_v$ the   ring of integers of $\K_v$, and by $\pi_v$ an uniformizer of~$\K_v$. 
Then the $p$-completion $\E_\K:=\Z_p\otimes E_K$ of $E_\K$ embeds diagonally, via $\iota$,  in $\U_p:=\prod_{v\in S_p} \U_v^1$, where $\U_v^1:=1+\pi_v \O_v$ is the group of principal units of $\K_v$. Observe that here $\U_p \simeq \Z_p^m$.
By $p$-adic class field theory (and due to the fact that $p\nmid h_\K$), the group $\G_p^{ab}$ is isomorphic to $\U_p/\iota(\E_\K)$. Then, assuming  Leopoldt conjecture for $\K$ at~$p$ (meaning here that $\iota$ is injective), the number field~$\K$ is~$p$-rational if and only if $\U_p/{\iota(\E_\K)}$ is without torsion.

\subsubsection{} \label{section212}
 
Observe that as $p$ is unramified in $\K/\Q$, we also get that  $p\nmid |\mu_\K|$, and as $p\nmid \#\G$, the character (as $\G$-module) of $\E_\K$ is equal to the character of $\Q_p\otimes(\Q\otimes E_\K) \simeq \Ind_{D_\infty}^\G \1$, 
where $D_\infty$ is the decomposition group  of an archimedean place in $\K/\Q$ and where $\1$ is the trivial character. In particular, $\E_\K$ is a submodule of the regular representation.

To be complete,   $\U_p$ is isomorphic to the regular representation (here $\U_v^1$ has no nontrivial root of unity).

\subsubsection{} \label{section213}

Let us fix an irreducible $\Q$-character $\chi$ of $\G$. 
Let   $\Q[\G]e_\chi \simeq {\rm M}_{n_\chi}(D)$ be the simple algebra of $\Q[\G]$ associated to $\chi$, where $D$ is a skew field of degree $s_\chi^2$ over its center (the integer $s_\chi$ is the Schur index of $\chi$).
Then $\chi=s_\chi \sum_{\psi|\chi}  \psi$, where the sum is taken over irreducible $\Q_p$-characters $\psi$ dividing~$\chi$ (here $p\nmid \#\G$). 

Let $E_\K^\chi$ be the $\chi$-component of the $\Q[\G]$-module $\Q\otimes E_\K$, then the character of  $E_\K^\chi$  is written as $t_\chi \chi$ for some $t_\chi \in \{0, \cdots, n_\chi\}$.
Given an irreducible $\Q_p$-character $\psi|\chi$, the integer $s_\chi t_\chi$ is then the $\psi$-rank $r_\psi(E_\K)$  of $\Q_p\otimes E_\K$.

\medskip

If $M$ is a $\Z_p[\G]$-module of finite type,  the $\psi$-rank $r_\psi(M)$ of $M$ is defined as  $\displaystyle{r_\psi(M):=\frac{1}{\deg(\psi)}\dim_{\fq_p}( M^\psi/(M^\psi)^ p)}$.

As seen before $r_\psi(E_\K)=r_\psi(\E_\K)$, obviously $r_\psi(\E_\K) \geq r_\psi(\iota(\E_\K))$, and Leopoldt conjecture is equivalent to the equality $r_\psi(\E_\K)=r_\psi(\iota(\E_\K))$ for every  $\chi$ and~$\psi$. 
Observe that one knows that $r_\psi(\iota(\E_\K))\geq 1$ when $r_\psi(\E_\K) \neq 0$ (see \cite{Jaulent}).
\begin{rema} When $\G$ is abelian, one has $r_\psi(\E_\K) \leq  1$.
 \end{rema}

As seen before, with all the assumptions,  the torsion of  $\U_p/\iota(\E_\K)$ is isomorphic to $\T_p$. Thus, $r_\psi(\T_p) \leq r_\psi(\E_\K)$. If for every $\psi|\chi$ the $\psi$-rank of $\U_p/\iota(\E_\K)$ is maximal, meaning $r_\psi(\T_p)=r_\psi(\E_\K)$, then necessarily, for every unit $x \in E_\K^\chi$ such that $x\equiv 1 (\mod  \ \p)$  for all $\p|p$, one must have $x\equiv 1 (\mod  \ \p^2)$ for all $\p|p$.

\begin{lemm} \label{lemma1.2}

If there exists an unit $u \in E_\K ^\chi$ such that $u\equiv 1 (\mod \ \p_0)$ but $u\nequiv 1 (\mod \ \p_0^2)$ for some $\p_0|p$, then $r_\psi(\T_p) <r_\psi(\E_\K)$ for some $\psi |\chi$.
\end{lemm}

\begin{proof}
Put  $x=u^{\N(\p_0)-1} \in E_\K^\chi$, where $\N(\p)=\# \O_\K/\p$. Observe that $x\equiv 1 (\mod \ \p)$ for every $\p|p$ (the extension $\K/\Q$ is Galois) but, easily, one also has $x\nequiv 1 (\mod \ \p_0^2)$. We conclude with the small discussion above.
\end{proof}

\subsection{The generalized $abc$-conjecture}

See  \cite{Vojta}.
If $I \subset \O_\K$ is an integral ideal, let us  denote by $\Rad(I)$ the following ideal:
$$\Rad(I)=\prod_{\p | I} \N(\p),$$
where the product is taken over prime ideal $\p$ dividing $I$ and where as usual $\N(\p)=\#\O_K/\p$ is the absolute norm of $\p$. 

The generalized $abc$-conjecture for  $\K$ states that for any $\varepsilon>0$, there exists a constant $C_{K,\varepsilon}>0$ such that the inequality :
$$\prod_{v}\max\{|a|_v,|b|_v,|c|_v\}\leqslant C_{K,\varepsilon}\left(\Rad(abc)\right)^{1+\varepsilon}$$
holds for all nonzero $a,b,c \in \O_K$ verifying $a+b=c$, $(a,b)=1$, where the product is taken over all absolute values of $\K$ and where $|\cdot|_v$ denotes the normalized norm of $\K_v$ (such that $\prod_v|x|_v=1$ for all $x\in \K^\times$).

Here we use it in the case where $b=u_2$ and $c=u_1$ are two distinct units of $\K$ and $a=u_1-u_2$ : for every $\varepsilon >0$, there exists a constant $C_{\K,\varepsilon}$ such that for all  $u_1\neq u_2 \in E_\K$, one has 
$$|\N(u_1-u_2)| \leq C_{\K,\varepsilon} \Rad((u_1-u_2))^{1+\varepsilon}.$$


\section{Proofs}

\subsection{} As explained in Introduction, some part of the proof is greatly inspired by \cite{Silverman}.

Let $\K/\Q$ be a Galois extension of degree $m$.
Consider the number field $\L:=\K(\zeta)$  where~$\zeta$ is a primitive $n$th-root of $1$. The extension $\L/\Q$ is Galois of degree $O(\varphi(n))$.  

\medskip

Let $T_n$ be the set of integers $j\in \{1,\cdots, n-1\}$ coprime to $n$. We denote by $\Phi_n$ the $n$th cyclotomic polynomial: $\displaystyle{\Phi_n(u)=\prod_{j\in T_n}(u-\zeta^j)}$.
The polynomial $\Phi_n$ is of degree $\varphi(n)$.
Thereafter, we will focus on integer $n$ such that $\varphi(n)\geq \frac{1}{2}n$.
 Recall Lemma 6 of~\cite{Silverman}: 
 $$\#\{n\leqslant X , \varphi(n) \geq \frac{1}{2} n\}\geq (\frac{6}{\pi^2}-\frac{1}{2})\ X+O(\log\ X).$$

\medskip

We start with the key lemma extending Lemma 5 of \cite{Silverman}.
 
\begin{lemm} \label{lemma_5} Let $u \in E_\K \cap \Ss$. Then there exists some $k \in \Z_{>0}$ such that
$$|\N(\Phi_n(u^k))|\geq \exp(c n),$$
for $n$ such that $\varphi(n)\geq \frac{1}{2} n$, where $c>0$ is a constant depending on $u$ and $k$.
\end{lemm}

\begin{proof}
As $u \in \Ss$, there exists an embedding $\sigma : \K \hookrightarrow \C$ such that $|\sigma(u)| \geq a >1$, for some real $a$.   Hence, for $k\in \Z_{>0}$, we get $|\sigma(u^k)|\geq a^k$, and then $|\sigma(u^k)-\zeta^{j}| \geq a^k-1$.

Let us choose an another embedding $\tau$. We want to give some "good" lower bound for $|\tau(u^k)-\zeta^j|$.
As $u\in \Ss$ there is only two situations.

\medskip

$\bullet$  If $|\tau(u)| <1$, then clearly for sufficiently  large $k$, we get $$|\tau(u^k)- \zeta^j|\geq 1-|\tau(u^k)|\geq \frac{1}{2}.$$

$\bullet$ If $|\tau(u)| >1$, for sufficiently large $k$, we get $|\tau(u^k)-\zeta^j| \geq  1$.

\medskip

Putting all of this together, we obtain
$$\N(\Phi_n(u^{k}))=\prod_{i=1}^m\prod_{j\in T_n}|\sigma_i(u^k) - \zeta^j|\geq \left((a^k-1)2^{-m+1}\right)^{\varphi(n)},$$
Consequently, by taking  sufficiently large $k$, we get that  for every $n$ with $\varphi(n)\geq \frac{1}{2} n$
$$\N(\Phi_n(u^{k}))\geq \exp(c n),$$ where the  $\sigma_i$'s are the embeddings of $\K$ in $\C$ and where 
 $c>0$ is some constant (depending on $u$, $k$ and $m$).
\end{proof}

\

Suppose now that $u \in E_\K$ is such that $$|\N(\Phi_{n}(u))|\geq \exp(c n),$$
for every $n$ such that $\varphi(n)\geq \frac{1}{2} n$ 
(which is always possible by Lemma \ref{lemma_5}).
 
Let us write $(u^n-1)=I_n J_n$, with $I_n$ and $J_n$ relatively prime and where if $\p |I_n$, then $\p^2 \nmid I_n$, and if $\p |J_n$ then $\p^2 |J_n$.
Then, if we write $u^n-1+1=u^n$, the generalized $abc$-conjecture implies that
$$|\N(u^n-1)| \ll_{\K,\varepsilon} \Rad(I_n J_n)^{1+\varepsilon} \ll_{\K,\varepsilon} \left(\N(I_n) \N(J_n)^{1/2}\right)^{1+\varepsilon}.$$
Hence, as $|\N(u^n-1)|=\N(I_n)N(J_n)$, we get
$$\N(J_n)^{1/2}\ll_{\K,\varepsilon} \N(I_n)^{\varepsilon}\N(J_n)^{\varepsilon/2} \ll_{\K,\varepsilon} |\N(u^n-1)|^{\varepsilon}, $$
and then $$\N(J_n) \ll_{\K,\varepsilon} |\N(u^n-1)|^{2 \varepsilon}.$$
Now let us  also write $(\Phi_n(u))=A_nB_n$, with $A_n$ and $B_n$ relatively prime and  where if $\p |A_n$, then $\p^2 \nmid A_n$, and if $\p |B_n$ then $\p^2 |B_n$. Of course, $B_n | J_n$, and then 
$$
 \N(B_n) \ll_{\K, \varepsilon} |\N(u^n-1)|^{2 \varepsilon}.
$$
Choose $\beta > 1$ such that $|\sigma_i(u)| \leq \beta$ for all $i$. Then
$$|\N(u^n-1)|\leq \prod_{i=1}^m(|\sigma_i(u)|^n+1)\leq 2^m(\beta^m)^n,$$
which implies $$\N(B_n) \ll_{\K, \varepsilon} 2^{2m \varepsilon}(\beta^{m})^{2n\varepsilon}.$$
Hence, $$\N(A_n)= \N(\Phi_n(u))/\N(B_n) \gg_{\K, \varepsilon}  \exp(n(c-2m\varepsilon \log \beta)).$$

We finally obtain:
\begin{prop} \label{prop2.2} If the generalized $abc$-conjecture holds 
 then for all $\varepsilon >0$, one has
 $$\N(A_n) \gg_{\K,\varepsilon}   \exp(n(c-2m\varepsilon \log \beta)),$$
 for every $n$ such that $\varphi(n) \geq \frac{1}{2} n$. 
\end{prop}

Take now $\varepsilon >0$ such that $\displaystyle{\varepsilon < \frac{c}{2m\log(\beta)}}$. Thanks  to Proposition \ref{prop2.2}, there exists $n_0 \in \Z_{>0}$ such that for all $n\geq n_0$, with $\varphi(n) \geq \frac{1}{2}n$, then $\N(A_n) > n^{m}$, where we recall that $m=[\K:\Q]$.  
Then, for each such~$n$, there exists  a prime ideal~$\p_n \subset \O_K$, dividing $A_n$ but not~$n$: indeed  if it was not the case then as $A_n$ is square free,  $A_n$ would divide $n$, which contradicts $\N(A_n) > n^m$.
Observe that $\p_n | (u^n-1)$ implies  $\N(\p_n) \leq 2^m \beta^{mn}$.

As $\p_n \nmid n$, the polynomial $X^n-1$ is separable over $\O_K/\p_n$. Thus $u$ is a simple root of $X^n-1=\prod_{d| n}\Phi_d(X)$ modulo $\p_n$ and, as $\p_n$ divides $\Phi_n(u)$, its order in $\left(\O_\K/\p_n\right)^\times$ is exactly $n$. Furthermore, $\p_n$ is a divisor of $A_n$, so $\p_n ^2$ does not divide $u^n-1$ (in other words $\p_n | I_n$). 

Let $p_n$ be the prime number such that $p_n \Z=\p_n \cap \Z$. 

In conclusion, we obtain:

\begin{prop} \label{prop_key}
Take $u\in E_\K$ as before. For each $n\geq n_0$ such that  $\varphi(n) \geq \frac{1}{2} n$,  there exists a prime ideal $\p_n \subset \O_\K$ such that 
\begin{itemize}
\item[$(i)$] $\p_n | \Phi_n(u)$ and $u^n\nequiv 1 (\mod \ \p_n^2)$,
\item[$(ii)$] $u$ is of order $n$ in $\left(\O_\K/\p_n\right)^\times$,
\item[$(iii)$] $\N(\p_n) \leq  \gamma^n$, for some $\gamma$ depending only on $\K$.
\end{itemize}
\end{prop}

By $(ii)$ of Proposition \ref{prop_key}, it follows that $\p_n=\p_{n'}$ if and only if $n=n'$.
 Observe that a set of primes $\p_n$ of size $Y$ gives at least $Y/m$ primes~$p_n$.
 
Now given $X \geq 1$,  let $n_1$ be the largest integer such that $\gamma^{n_1} \leq X$. Assume $X$ sufficiently large to ensure  $n_0 \leq n_1$. Then, for each $n \in [n_0, n_1]$ 
such that $\varphi(n)\geq \frac{1}{2}n$,
there exists a prime ideal $\p_n \subset \O_\K$   for which $u^n\equiv 1 (\mod \ \p_n)$ and $u^n\nequiv 1 (\mod \ \p_n^2)$.  Note that $p_n \leq \N(\p_n) \leq \gamma^n \leq \gamma^{n_1}\leq  X$.
Thereby:
\begin{eqnarray*}
&&\frac{1}{m} \#\{n, \ n_0\leq n\leqslant n_1 , \ \varphi(n)\geq \frac{1}{2} n\}\\
&&\leq \#\{p_n \leq X, p_n \text{ prime }|\ \exists~ \p_n\in \O_k, \p_n|p_n , u^n\equiv 1 (\mod \ \p_n) \text{ and } u^n\nequiv 1 (\mod \ \p_n^2)\}.
\end{eqnarray*}

In conclusion, one has found at least  $c \ \log X$  prime numbers $p_n \leq X$ satisfying  $(i)$ of Proposition \ref{prop_key}  for some $\p_n|p_n$.

\subsection{} {\it Proof of Theorem \ref{theo_main}}. 
Let $\chi$ be an irreducible $\Q$-character of $\G$ such that there exists some $u\in E_\K^\chi \cap \Ss$.
By the previous section, there exists $k\geq 1$ such that $u^{kn}\equiv 1 (\mod \ \p_n)$ and $u^{kn}\nequiv 1 (\mod \ \p_n^2)$ for at least $c\ \log X$ prime numbers  $p_n \leq X$ (where $\p_n|p_n$). We conclude with Lemma \ref{lemma1.2} (after forgetting the prime numbers smaller than $p_0$).

\medskip

{\it Proof of the Corollary}.

Observe first that, in the two cases, the Leopoldt conjecture holds and the field $\K$ contains some unit in $\Ss$ (see Remark \ref{rema_Pisot}). Take $p>p_0$.
The choice of the character is the following : if $\K$ is real quadratic, let $\chi=\psi$ be  the  nontrivial character of~$\G$ ; if $\K/\Q$ is an imaginary $S_3$-extension, let $\chi$ be the irreducible $\Q$-character of $\G$ of degree $2$ (observe that  $\chi=\psi$ is also $\Q_p$-irreducible).
Then  $\Q\otimes E_\K =E_\K^\chi$,   $r_\psi(E_\K)=1$, and $\T_p=\T_p^\psi$. Therefore by Theorem \ref{theo_main},  $\T_p=\{1\}$ for at least $c\ \log X$ prime numbers $p\leq X$.


\end{document}